\documentclass[preprint,10pt]{elsarticle}
\usepackage{times}
\usepackage{cite}
\usepackage{amsthm,amsmath,amssymb,amsfonts}
\usepackage{algorithmic}
\usepackage{graphicx}
\usepackage{textcomp}
\usepackage{mfl}
\usepackage{mathtools}
\usepackage{mathtools}

\usepackage{amsmath}
\usepackage{amssymb}

\usepackage{tikz}
\usetikzlibrary{arrows,chains,matrix,positioning,scopes}
\def\BibTeX{{\rm B\kern-.05em{\sc i\kern-.025em b}\kern-.08em
    T\kern-.1667em\lower.7ex\hbox{E}\kern-.125emX}}

\makeatletter
\tikzset{join/.code=\tikzset{after node path={%
\ifx\tikzchainprevious\pgfutil@empty\else(\tikzchainprevious)%
edge[every join]#1(\tikzchaincurrent)\fi}}}
\makeatother

\tikzset{>=stealth',every on chain/.append style={join},
         every join/.style={->}}
\tikzstyle{labeled}=[execute at begin node=$\scriptstyle,
   execute at end node=$]

\newtheorem{Def}{Definition}
\newtheorem{Thm}{Theorem}
\newtheorem{Rmk}[Thm]{Remark}
\newtheorem{Exm}[Thm]{Example}
\newtheorem{Cor}[Thm]{Corollary}
\newtheorem{Lem}[Thm]{Lemma}
\newtheorem{Pro}[Thm]{Proposition}

\usepackage{color,dsfont}

\theoremstyle{plain}

\theoremstyle{definition}
\usepackage[]{graphicx}

\begin{document}

\begin{frontmatter}

%% Title, authors and addresses

%% use the tnoteref command within \title for footnotes;
%% use the tnotetext command for theassociated footnote;
%% use the fnref command within \author or \address for footnotes;
%% use the fntext command for theassociated footnote;
%% use the corref command within \author for corresponding author footnotes;
%% use the cortext command for theassociated footnote;
%% use the ead command for the email address,
%% and the form \ead[url] for the home page:
%% \title{Title\tnoteref{label1}}
%% \tnotetext[label1]{}
%% \author{Name\corref{cor1}\fnref{label2}}
%% \ead{email address}
%% \ead[url]{home page}
%% \fntext[label2]{}
%% \cortext[cor1]{}
%% \address{Address\fnref{label3}}
%% \fntext[label3]{}

\title{Frame definability in finitely-valued modal logics}

%% use optional labels to link authors explicitly to addresses:
%% \author[label1,label2]{}
%% \address[label1]{}
%% \address[label2]{}

\author{Guillermo Badia}

\address{
School of Historical and Philosophical Inquiry\\
 University of Queensland\\ 
 Brisbane, St Lucia, QLD 4072, Australia\\ 
 %Institute of Philosophy and Scientific Method\\ Johannes Kepler University Linz\\ Linz, Altenberger Strasse 69, Austria \\
\texttt{guillebadia89@gmail.com} }

\author{Xavier Caicedo}

\address{Department of Mathematics\\ University of los Andes \\ Carrera 1 \# 18A - 12.\\ 11171 Bogot\'a, Colombia\\ \texttt{xcaicedo@uniandes.edu.co}}

\author{Carles Noguera}

\address{Department of Information Engineering and Mathematics\\University of Siena\\
          San Niccol\`o, via Roma 56\\53100 Siena, Italy\\
           \texttt{carles.noguera@unisi.it}}

\begin{abstract}
In this paper we study frame definability in finitely-valued modal logics and establish two main results via suitable translations: (1) in finitely-valued modal logics one cannot define more classes of frames than are already definable in classical modal logic (cf.~\citep[Thm.~8]{tho}), and (2)  a large family of finitely-valued modal logics define exactly the same classes of frames as classical modal logic (including modal logics based on finite Heyting and \MV-algebras, or even \BL-algebras). In this way one may observe, for example, that the celebrated Goldblatt--Thomason theorem applies immediately to these logics. In particular, we obtain the central result from~\citep{te} with a much simpler proof  and answer one of the open questions left in that paper. Moreover, the proposed translations allow us to determine the computational complexity of a big class of finitely-valued modal logics.

\end{abstract}

\begin{keyword} many-valued logics \sep modal logics \sep frame definability \sep finite lattices
%% keywords here, in the form: keyword \sep keyword

%% PACS codes here, in the form: \PACS code \sep code

%% MSC codes here, in the form: \MSC code \sep code

\MSC 03C95 \sep 03B52 \sep 03B50
%% or \MSC[2008] code \sep code (2000 is the default)

\end{keyword}

\end{frontmatter}

\section{Introduction}

Propositional modal logic is, at the level of frames, famously incomparable (in terms of expressive power) with first-order logic. Indeed, there are classes of frames definable in first-order logic that are not modally definable and viceversa. An example of the former phenomenon is the class of frames axiomatizable by the first-order sentence $\forall x \,\exists y\,(Rxy \wedge Ryy)$, and an example of the latter is the class of frames axiomatized by ``L\"ob's formula''  $\Box (\Box p\rightarrow p) \rightarrow \Box p$ which defines frames with transitive relations where the converse relation is well-founded~\citep[pp.~33--34]{van3}.

The original Goldblatt--Thomason theorem~\citep[Thm.~8]{gold} provides a model-theoretic characterization of modal axiomatizability for elementary classes of frames in terms of closure under taking generated subframes, disjoint unions, bounded morphic images, and reflection of  ultrafilter extensions. Even though the proof in~\citep{gold} is algebraic in spirit (with a detour via the usual extension of Stone duality), there are well-known ways of obtaining the result by pure model-theoretic methods \citep{van}. Each closure condition in the theorem is necessary (see~\citep[p.~6]{van}). Furthermore, the condition of elementarity of the class can be relaxed to closure under ultrapowers. There are also results, already contained in~\citep{gold}, using more complicated constructions which characterize any modally axiomatizable class of frames.

Many-valued modal logics, i.e.\ expansions with modalities of non-classical propositional logics with an intended many-valued semantics, have been around at least since Kristen Segerberg studied three-valued modal logics in~\citep{Segerberg}. The topic gained momentum with Melvin Fitting's work when he axiomatized in~\citep{Fitting:MVModal,Fitting:MVModalII} the relational semantics for these logics based on Kripke frames with finitely-valued propositional evaluations in each world (possibly, also, finitely-valued accessibility relations between worlds), and proposed a natural interpretation of modalities capitalizing on the lattice structure of the semantics of the base logic. This proposal inspired, in particular, a long and nowadays quite lively stream of research in fuzzy modal logics (see e.g.~\citep{bou,Caicedo-Rodriguez:OrderModal,Caicedo-Rodriguez:BimodalGodel,Cintula:GeneralPossibleWorld,Hajek:1998,Hajek:FuzzyModalLogicS5,HansoulTeheux:ModalLukasiewicz,Mart,Vidal:PhD,Vidal-Esteva-Godo:ModalProduct}). As a part of this research, Bruno Teheux~\citep{te} has established an analogue of the Goldblatt--Thomason theorem for modal \L ukasiewicz logics determined by finitely-valued Kripke models over crisp frames (i.e.\ with a two-valued accessibility relation).

The goal of the present paper is to investigate frame definability in the many-valued context in a general approach that encompasses Teheux's results. We prove two main results: (1) each class of crisp frames definable in a finitely-valued modal logic is already definable in classical modal logic, and (2) for a large family of finitely-valued modal logics, the converse inclusion also holds, that is, their definable crisp frames {\em coincide} with those definable in classical modal logic. We proceed via translations of the many-valued modal logic into classical modal logic and back (inspired by the Kolmogorov--Glivenko translation) which preserve crisp frames. Furthermore, for these finitely-valued modal logics, our translations ensure that their computational complexity  coincides with that of their two-valued counterparts.

The first result generalizes a (little known) work by Steven K.~Thomason~\citep{tho} in which he translated finitely-valued modal logics into two-valued modal logics. His approach mostly stayed at the level of frames and did not provide an explicit recursive definition. Moreover, his result was restricted to a class of logics in a language with standard connectives and truth constants. We propose a translation close to Thomason's, but based on models, with an explicit recursive definition, and free from the mentioned syntactical restrictions. The converse inclusion that we present in our second main result has not yet been considered in the literature, as far as we know, and it applies to any modal logic based on a finite lattice algebra that can interpret a Boolean algebra (thus, including modal logics based on finite Heyting and \MV-algebras, and even \BL-algebras). Therefore, even though Teheux's result had been presented as a generalization of the original theorem from~\citep{gold}, it actually {\em follows} from the classical Goldblatt--Thomason theorem and our results. This also answers, as we will see, an open problem left in~\citep{te}.
Incidentally, we can also obtain the extension of the Goldblatt--Thomason theorem for predicate finitely-valued modal logic using the work from~\citep{zho} for the classical setting.

The paper is arranged as follows:  \S \ref{pre} succintly presents the necessary preliminaries regarding the syntactical and semantical setting of the paper. \S \ref{trans} introduces our translation from finitely-valued to classical modal logic, shows that it has the intended semantic behavior. \S \ref{main} contains the mentioned two main results of the paper: in \S \ref{ModalImpliesClassical} we prove that many-valued modal definability implies classical modal definability and  \S \ref{ClassicalImpliesModal} shows that the implication can be reversed whenever the many-valued modal logic is based on a finite lattice algebra that interprets a Boolean algebra. \S \ref{comp} uses the results in the previous sections to prove that the problems of validity and consequence from finite sets of premises in the considered many-valued modal logics have exactly the same computational complexity as their classical counterparts. Finally, \S \ref{con} offers some concluding remarks and lines for further research.

\section{Preliminaries}\label{pre}
Let $\alg{A} = \tuple{\land^\alg{A},\lor^\alg{A},\ldots}$ be an arbitrary {\em finite} (henced bounded) lattice possibly expanded with further operations (which from now on we will call a \emph{lattice algebra}). When convenient, we will denote the top and the bottom element of $\alg{A}$ respectively as $1_\alg{A}$ and $0_\alg{A}$. In particular, we will refer to the following instances:
\begin{itemize}
\item A {\em Boolean algebra} is a lattice algebra $\alg{A} = \tuple{\land^\alg{A},\lor^\alg{A},\neg^\alg{A}}$ in which the lattice is distributive and $\neg$ is the complement operation (that is, for each element $a\in A$, $a \lor^\alg{A} \neg^\alg{A} a = 1_\alg{A}$ and $a \land^\alg{A} \neg^\alg{A} a = 0_\alg{A}$).
\item A {\em pseudocomplemented lattice} is a lattice algebra $\alg{A} = \tuple{\land^\alg{A},\lor^\alg{A},\neg^\alg{A}}$ in which for each $a \in A$, $\neg^\alg{A} a = \max \{b \in A \mid a \land b = 0_\alg{A}\}$.
\item A {\em Stone algebra} is a pseudocomplemented lattice $\alg{A} = \tuple{\land^\alg{A},\lor^\alg{A},\neg^\alg{A}}$ in which the lattice is distributive and, for each $a \in A$, $\neg^\alg{A} a \lor \neg^\alg{A} \neg^\alg{A} a = 1_\alg{A}$.
\item A {\em Heyting algebra} is a lattice algebra $\alg{A} = \tuple{\land^\alg{A},\lor^\alg{A},\to^\alg{A}}$ such that, for each $a,b \in A$, $a \to^\alg{A} b = \max \{c \in A \mid a \land^\alg{A} c \leq b\}$.
\item An {\em \MV-algebra} is a lattice algebra $\alg{A} = \tuple{\land^\alg{A},\lor^\alg{A},\conj^\alg{A},\to^\alg{A}}$ such that
\begin{itemize}
\item $\conj^\alg{A}$ is commutative, monotonic w.r.t.\ the lattice order, and has $1_\alg{A}$ as neutral element,
\item for each $a,b \in A$, $a \to^\alg{A} b = \max \{c \in A \mid a \conj^\alg{A} c \leq b\}$,
\item for each $a,b \in A$, $(a \to^\alg{A} b) \lor^\alg{A} (b \to^\alg{A} a) = 1_\alg{A}$,
\item and for each $a,b \in A$, $a \lor^\alg{A} b = (a \to^\alg{A} b) \to^\alg{A} b$.
\end{itemize}
\end{itemize}

The formulas of the modal language $\fm_\alg{A}^{\Diamond \Box}(\tau)$ are built from a denumerable set of propositional variables $\tau$ by means of the binary connectives $\land$ and $\lor$, an $n$-ary connective for each additional $n$-ary operation of $\alg{A}$, and the unary connectives (modalities) $\Diamond$ and $\Box$. 

In particular, if \alg{2} is the two-element Boolean algebra, we may think of the formulas of $\fm_\alg{2}^{\Diamond \Box}(\tau)$ as the usual classical modal language.

%$$ \f := p \ | \ \neg \f    \ | \  \f \wedge \p   \ | \ \f \vee \p    \ | \ \Box \f \ | \ \Diamond \f.$$
%In this setting, as usual, we may let $\bot $ stand for any fixed contradiction of the form $\f \wedge \neg \f$.

We will work with \emph{crisp} frames $\model{F}= \tuple{W,R}$ as in classical modal logic, i.e.\ $W$ is a non-empty set (whose elements are called \emph{worlds}) and $R \subseteq W^2$ is a binary relation (called \emph{accessibility relation}). A Kripke \alg{A}-valued model $\model{M}$ is defined as a pair $\tuple{\model{F},V}$, where $\model{F}= \tuple{W, R}$ is a frame and $V \colon \tau \times W \longrightarrow A$ is mapping called a \emph{valuation}; we say that $\model{M}$ is \emph{based on $\model{F}$}. Given such a model, for each $w \in W$ and each formula $\f \in \fm_\alg{A}^{\Diamond \Box}(\tau)$, we inductively define the truth-value $\semvalue{\f}^\model{M}_w$ as:

\begin{center}
\begin{tabular}{rclll}
$\semvalue{p}^\model{M}_w$ & $=$ & $V(p,w),$ & if $p \in \tau$
\\[0.5ex]
%$ \semvalue{\0}^\model{M}_w$ & $=$ & $\0^\alg{A}\!,$ &
%\\[0.5ex]
%$ \semvalue{\1}^\model{M}_w$ & $=$ & $1_\alg{A}\!,$ &
%\\[0.5ex]
%$ \semvalue{\p_1 \land \p_2}^\model{M}_w$ & $=$ & $ \semvalue{\p_1}^\model{M}_w \land^\alg{A} \semvalue{\p_2}^\model{M}_w\!,$ &
%\\[0.5ex]
%$ \semvalue{\p_1 \lor \p_2}^\model{M}_w$ & $=$ & $ \semvalue{\p_1}^\model{M}_w \lor^\alg{A} \semvalue{\p_2}^\model{M}_w\!,$ &
%\\[0.5ex]
$ \semvalue{\circ(\p_1, \ldots, \p_n)}^\model{M}_w$ & $=$ & $\circ^\alg{A}(\semvalue{\p_1}^\model{M}_w, \ldots, \semvalue{\p_n}^\model{M}_w),$ & for each $n$-ary connective $\circ$,
\\[0.5ex]
$ \semvalue{\Diamond \p}^\model{M}_w$ & $=$ & $\sup_{\leq_\alg{A}}
\{\semvalue{\p}^\model{M}_v \mid Rwv\},$ &
\\[0.5ex]
$ \semvalue{\Box\p}^\model{M}_w$ & $=$ & $\inf_{\leq_\alg{A}}
\{\semvalue{\p}^\model{M}_v \mid Rwv \}.$ &
\end{tabular}
\end{center}

A formula $\f$ from  $\fm_\alg{A}^{\Diamond \Box}(\tau)$ is said to be \alg{A}-\emph{valid} in a frame $\model{F}= \tuple{W, R}$ (in symbols, $\model{F} \models_\alg{A} \f$)  if for any \alg{A}-valued model $\model{M}$ based on $\model{F}$, $\semvalue{\f}^\model{M}_w = 1_\alg{A}$ for every world $w \in W$ (alternatively, we say that $\f$ is \emph{globally true} in $\model{M}$). Furthermore, a set of formulas $\Phi$ from $\fm_\alg{A}^{\Diamond \Box}(\tau)$ \emph{modally \alg{A}-defines} a frame class $\mathbb{F}$  if $\mathbb{F}$ contains exactly those frames where every $\f \in \Phi$ is \alg{A}-valid (if $\Phi = \{\f\}$, we say that \emph{$\f$ modally \alg{A}-defines $\mathbb{F}$}) (cf.~\citep[Definition 2.2 and 2.3]{te}).  Similarly, we may say that a class $\mathbb{F}$ of frames is \emph{modally \alg{A}-definable} if there is a set of formulas $\Phi$ such that modally \alg{A}-defines $\mathbb{F}$. Finally, given a frame class $\mathbb{F}$, we define a consequence relation in the following way: for each $\Gamma \cup \{\f\} \subseteq \fm_\alg{A}^{\Diamond \Box}(\tau)$, we write $\Gamma \vDash_{\mathrm{Log}(\mathbb{F},\alg{A},\tau)} \f$ iff for each $\mathfrak{F} \in \mathbb{F}$ and each \alg{A}-valued model $\model{M}$ based on $\model{F}$ we have that $\f$ is globally true in $\model{M}$ whenever all formulas from $\Gamma$ are  globally true in $\model{M}$. Thus, $\mathrm{Log}(\mathbb{F},\alg{A},\tau)$ can be called the {\em global modal $\alg{A}$-valued logic given by $\mathbb{F}$.}

% Finally, a class $\K$ of \emph{pointed models} (i.e., models of the form $\langle\model{M}, w\rangle$) is  \emph{modally \alg{A}-definable} if there is a set of formulas $\Phi$ s.t.  the pointed models that make every $\f \in \Phi$ true (i.e., $\semvalue{\f}^\model{M}_w = 1^\alg{A}$) are exactly those in $\K$.   

Observe that in the case $\alg{A}\cong \alg{2}$ we retrieve the standard definitions from classical modal logic. In this case, we use the standard notation $\tuple{\model{M}, w} \models \f$ to signify that $\semvalue{\f}^\model{M}_w =1_\alg{A}$.

\section{Translating finitely-valued modal logics into classical modal logics}\label{trans}

In this section, we provide a translation of formulas of a many-valued modal logic into formulas of standard two-valued modal logic. In contrast to  \citep{tho},  we give an explicit inductive definition of the translation already at the level of models.
%Our work is similar to \citep{tho}, except that we give an explicit, more fine-grained, inductive definition of the translation already at the level of models, while Thomason stays largely at the level of frames. Furthermore, we introduce a novel translation for the case where the frames have many-valued accessibility relations, a case that was not dealt with in \citep{tho}.

Given a finite lattice algebra $\alg{A}$ and a denumerable set of variables $\tau = \{p_1, p_2, \ldots\}$, we define $\tau^* =\bigcup\limits_{i \geq 1}\{p_i^a \mid a \in A\} $. Now,  we define translations $T^\alg{a}$ from $\fm_\alg{A}^{\Diamond \Box}(\tau)$ into $\fm_\alg{2}^{\Diamond \Box}(\tau^*)$, for each  element $a \in A$, by simultaneous induction as follows:
 \begin{center}

\begin{tabular}{rcll}
%$T^a(\overline{b})$ & $:=$ & 
%$\begin{cases*}
%      \bot & if $a\neq b$ \\
%     p_0^b      & otherwise
%    \end{cases*}$ &
%\\[0.5ex]
$T^a(p_i)$ & $=$ & 
$p_i^a$  \,\,\,\, ($i \geq 1$)&
\\[0.5ex]
%$T^a(\0)$ & $=$ & $\0$
%\\[0.5ex]
%$T^a(\1)$ & $=$ & $\1$
%\\\vspace{0.4ex}
%$T^a(\p_1 \land \p_2)$ & $=$ & $\bigvee\limits_{\substack {b_1, b_2 \in A \\b_1\land^\alg{A} b_2 = a}}(T^{b_1}(\p_1) \wedge  T^{b_2}(\p_2))$ & 
%\\\vspace{0.4ex}
%$T^a(\p_1 \lor \p_2)$ & $=$ & $\bigvee\limits_{\substack {b_1, b_2 \in A \\b_1\lor^\alg{A} b_2 = a}}(T^{b_1}(\p_1) \wedge  T^{b_2}(\p_2))$ & 
%\\\vspace{0.4ex}
$T^a({\circ}(\vectn{\p}))$ & $=$ & $\bigvee\limits_{\substack {b_1,  \ldots, b_n \in A \\{\circ}^\alg{A}(b_1, \ldots, b_n) = a}}(T^{b_1}(\p_1) \wedge \ldots\wedge T^{b_n}(\p_n))$ & 
\\\vspace{0.4ex}
$T^a(\Diamond \p)$ & $=$ & $(\bigvee\limits_{\substack {k \leq |A| \\b_1  \ldots b_k \in A \\ b_1 \vee^\alg{A} \ldots \vee^\alg{A} b_k = a}} \bigwedge\limits_{i=1}\limits^{k} \Diamond  T^{b_i}(\p) ) \wedge \Box ( \bigvee\limits_{\substack {b \in A \\{b \leq a}}}     T^{b}(\p))$
\\\vspace{0.4ex}
$T^a(\Box \p)$ & $=$ & $(\bigvee\limits_{\substack {k \leq |A| \\b_1,  \ldots, b_k \in A \\ b_1 \wedge^\alg{A} \ldots \wedge^\alg{A} b_k = a}} \bigwedge\limits_{i=1}\limits^{k} \Diamond  T^{b_i}(\p) ) \wedge \Box ( \bigvee\limits_{\substack {b \in A \\{a \leq b}}}     T^{b}(\p)). $

\end{tabular}
\end{center}

Furthermore, given any \alg{A}-valued model $\model{M} = \tuple{\model{F},V}$ based on a frame $\model{F}=\tuple{W, R}$ for $\fm_\alg{A}^{\Diamond \Box}(\tau)$, we define a \alg{2}-valued model $\model{M}^* $ for  $\fm_\alg{2}^{\Diamond \Box}(\tau^*)$ based on the same frame and with a valuation $V^*$ defined as follows:
\begin{center}
$V^*(p_i^a, w) = 1$ iff $V(p_i, w) = a$, for each $i \geq 1$, each $w \in W\!,$ and each $a \in A$.
\end{center}
Given this,  it is not hard to see that the translation $T^a(\f)$ simply rewrites in the classical modal language the conditions for $\f$ to take the value $a$; more precisely, by induction on the complexity of formulas, we can easily prove the following Switch Lemma:

\begin{Lem}[Switch Lemma]
Let $\alg{A}$ be a finite lattice algebra, $\tau$ a denumerable set of variables, and\/ $\model{M}$ an \alg{A}-valued model based on a frame $\model{F}$. For each formula $\f \in \fm_\alg{A}^{\Diamond \Box}(\tau)$, and each world $w$, we have:   $\semvalue{\f}^\model{M}_w = a$ iff\/ $\tuple{\model{M}^*, w} \models T^a(\f)$. 

\end{Lem}
%\begin{proof} This can be done by induction on the complexity of $\f$. 
%For example, if $\f = \overline{b} $ ($b \in A$), and $\semvalue{\overline{b}}^\model{M}_w = a$, then  $b=a$, so $T^a(\overline{b}) = p_0^b$, but also  $\langle\model{M}^*, w\rangle \models p_0^b$, as desired. On the other hand, if $\langle\model{M}^*, w\rangle \models T^a(\overline{b})$, $b=a$, so $\semvalue{\overline{b}}^\model{M}_w = a$, as desired. 
%We leave  the rest as an exercise for the interested reader. 
% \end{proof}
 
%\begin{Lem}
%Let $\alg{A}$ be a finite lattice, $\tau$ a denumerable set of variables. Then, for each \alg{2}-valued model $\model{N}$  for the language $\fm_\alg{2}^{\Diamond \Box}(\tau^*)$ that makes $T^*$ globally true, there is an $\alg{A}$-valued model $\model{M}$ for the language $\fm_\alg{A}^{\Diamond \Box}(\tau)$ such that $\model{M}^* = \model{N}$.
%\end{Lem}
%
%\begin{proof}
%Assume that $\model{N} = \tuple{\model{F},V}$, where $\model{F}=\tuple{W, R}$ and $V \colon \tau^* \times W \to \{0,1\}$. Then, we take $\model{M} = \tuple{\model{F},U}$ with $U \colon \tau \times W \to A$ defined as: $U(p_i,w) = a$ iff $V(p_i^a,w) =1$, for each $p_i \in \tau$, $w \in W$, $a \in A$. Then, it is clear that $\model{M}^* = \model{N}$.
%\end{proof}

We can easily obtain the following axiomatization result for classical models of the form $\model{M}^*$:

\begin{Lem}\label{goodlem}
 Let $\alg{A}$ be a finite lattice algebra and $\tau$ a set of variables.  Consider the following set $T^*(\tau)\subseteq \fm_\alg{2}^{\Diamond \Box}(\tau^*)$ of (modality-free) formulas:

 \[
 \bigvee_{\substack{a \in A }}p_i^a, \,\,  \neg (p_i^a\land p_i^b) \,\,\,\,\,\, \,\,\,\,\,\,  (a, b \in A, a \neq b, p_i \in \tau). 
 \]

Then:
 
\begin{enumerate}
\item If\/ $\model{M}$ is an \alg{A}-valued model, then the formulas of $T^*(\tau)$ are true in every world of the \alg{2}-valued model $\model{M}^*$.
\item If\/  $\model{N}$ is a \alg{2}-valued model  for the language $\fm_\alg{2}^{\Diamond \Box}(\tau^*)$ that satisfies in each world all the formulas of\/  $T^*(\tau)$, we define a model\/ $\model{M}$ for the language $\fm_\alg{A}^{\Diamond \Box}(\tau)$ as follows:
 \begin{itemize}
 \item $\model{M}$ is based on the same frame as $\model{N}$,
 \item $V^{\model{M}} (p_i, w) = a$ iff  $\tuple{\model{N}, w} \models p_i^a$.
\end{itemize}
Then, $\model{N} = \model{M}^*$.
\end{enumerate}
\end{Lem}
Observe that if $\tau$ is finite,  then $T^*(\tau)$ in Lemma \ref{goodlem} is finite as well.

%From these two lemmas, we immediately obtain the following translation result:

%\begin{Pro}\label{translation-consequence}
%Let $\alg{A}$ be a finite lattice algebra, $\tau$ a denumerable set of variables, and\/ $\mathbb{F}$ a class of frames. Then, for each $\Gamma \cup \{\f\} \subseteq \fm_\alg{A}^{\Diamond \Box}(\tau)$, we have:
%\begin{center}
%%\item $ \vDash_{\mathrm{Log}(\mathbb{F},\alg{A},\tau)} \f$ \,\, iff \,\, $ T^* \vDash_{\mathrm{Log}(\mathbb{F},\alg{2},\tau^*)} T^1(\f),$
%$\Gamma \vDash_{\mathrm{Log}(\mathbb{F},\alg{A},\tau)} \f$ \,\, iff \,\, $T^1[\Gamma]  \cup T^*(\tau) \vDash_{\mathrm{Log}(\mathbb{F},\alg{2},\tau^*)} T^1(\f),$
%%\item $\Gamma \vdash_{l\mathrm{Log}(\mathbb{F},\alg{A},\tau)} \f$ \,\, iff \,\, $T^1[\Gamma] \vdash_{l\mathrm{Log}(\mathbb{F},\alg{2},\tau^*)} T^1(\f)$
%\end{center}
%where $\vDash_{\mathrm{Log}(\mathbb{F},\alg{A},\tau)}$ and\/ $\vDash_{\mathrm{Log}(\mathbb{F},\alg{2},\tau^*)}$ are  global consequence relations and $T^1[\Gamma] =\{T^1(\f)\mid \f \in \Gamma\}$.
%\end{Pro}
%
%
%Then, from Proposition~\ref{translation-consequence} and the results from~\citep{chen1994}, the problem of consequence from  a variable finite set of premises in many-valued modal logics over all crisp frames is decidable and, moreover, is in EXPTIME.  

\begin{Rmk}\em{
It is worth observing that the translation we have presented in this section allow us to give very quick proofs of certain  properties of many-valued modal logics on  crisp  frames. For example,  both compactness and the finite model property are inherited from two-valued modal logic.}

\end{Rmk}

\begin{Rmk}\em{Observe that the translation presented in this section can be adapted to cover a bit more than just logics based on finite lattice algebras. For instance, we can cover the case in which \alg{A} is infinite but $n$-compact in the sense that, for each $X \subseteq A $, $\bigwedge X = a_1 \wedge \dots \wedge a_n$ and $\bigvee X = b_1 \vee \dots \vee b_n$ for some $a_1, \dots, a_n, b_1, \dots, b_n \in A$
(for examples of these lattices, see \citep[Ex.~6]{C-TD-M 2019}}).
\end{Rmk}

\section{Modal frame definability}\label{main}
 
In this section, we will establish our main results. First, we will see that in an \alg{A}-valued modal logic we cannot define more  classes of crisp frames than are already definable in classical modal logic. 
Second, for a wide class of algebras, the converse also follows, namely, any class of frames  which is definable in two-valued modal logic will be modally \alg{A}-definable. 
 
\subsection{Modal \alg{A}-definability implies modal \alg{2}-definability}\label{ModalImpliesClassical}
Recall  that, for a modal formula $\f$ of $\fm_\alg{2}^{\Diamond \Box}(\tau^*)$, we have the classical property \citep[Prop.~4.3]{rij2} that the truth of $\f$ in any pointed  model $\tuple{\model{M}, w}$  depends only on the restriction  (denoted by $\model{M}| n$) of $\tuple{\model{M}, w}$ to worlds that can be reached from $w$ through $R$ in at most $n$ steps, where $n=\mathrm{rank} (\f)$, i.e., the modal rank of $\f$ (\citep[Def.~4.2]{rij2}). Observe that $\mathrm{rank} (\f) =  \mathrm{rank}(T^a(\f))$, since our translation does not increase the modal rank. Then, for such a formula, the Switch Lemma can be extended to:
\begin{center}
$\semvalue{\f}^\model{M}_w = a$ \ iff \ $\tuple{\model{M}^*, w} \models T^a(\f)$ \ iff \ $\tuple{\model{M}^*|n, w} \models T^a(\f)$.
\end{center}

\begin{Thm}[\citep{tho}, Theorem 8]\footnote{The methods in~\citep{tho} are substantially more complex than ours, which is why we produce our own proof for the benefit of the reader.}\label{t:AdefImplies2def} Let \alg{A} be a finite lattice algebra, let $\f$ be a formula  from $\fm_\alg{A}^{\Diamond \Box}(\tau)$ (assume w.l.o.g.\ that $\tau = \{p_1, \ldots, p_n\}$ is the finite set of variables that appear in the formula), and let\/ $\mathbb{F}$ be a class of frames. Then, $\f$ modally \alg{A}-defines $\mathbb{F}$ iff\/  $\mathbb{F}$ is modally $\alg{2}$-defined in $\fm_\alg{2}^{\Diamond \Box}(\tau^*)$ by

$$\f^*:= (\bigvee_{m \leq \mathrm{rank}(T^{1_\alg{A}}(\f))} \neg \Box^m ( \bigwedge T^*(\tau) ) )  \vee T^{1_\alg{A}}(\f). $$
 \end{Thm}
\begin{proof}  ($\Rightarrow$): Assume first that $\f$ modally \alg{A}-defines the frame  class $\mathbb{F}$. Our goal is then to show  the following two claims:
\begin{itemize}
\item[(1)] the formula  $\f^*$ is \alg{2}-valid in every frame from $\mathbb{F}$, and
\item[(2)] every frame where $\f^*$ is \alg{2}-valid belongs to the class $\mathbb{F}$.
\end{itemize}
To see (1), take any $\alg{2}$-valued model $\model{M}$ for  $\fm_\alg{2}^{\Diamond \Box}(\tau^*)$ based on a frame  $\mathfrak{F}$ from $\mathbb{F}$. $\f^*$ is a material implication, so assume that at an arbitrary world $w$ of $\model{M}$  all the antecedents of $\f^*$ hold.
Then, let $k = \mathrm{rank}(T^{1_\alg{A}}(\f))$ and consider the pointed model $\tuple{\model{M}|k, w}$. The theory 
 \[
 \bigvee_{\substack{a \in A }}p_i^a, \ \bigwedge_{\substack{a, b \in A \\ a \neq b}} \neg (p_i^a\land p_i^b) \,\,\,\,\,\, \,\,\,\,\,\,   (1\leq i \leq n)
 \]
globally holds in $\tuple{\model{M}|k, w}$ since that is what the hypothesis of all the antecedents of $\f^*$ holding means. Now take any model based on $\mathfrak{F}$ of the form $\model{N}^*$  for $\fm_\alg{2}^{\Diamond \Box}(\tau^*)$ such that $\tuple{\model{M}|k, w} \cong \tuple{\model{N}^*|k, w}$ corresponding to some $\model{N}$ for $\fm_\alg{A}^{\Diamond \Box}(\tau)$.  There is always one such model: e.g.\ interpret the predicates of $\tau^*$ for the worlds in  $\tuple{\model{M}|k, w}$ as in that model and in every other world let $p_i^{1_\alg{A}}$ hold and $p_i^a$ fail, for each $1\leq i \leq n$ and each $a \neq 1_\alg{A}$ . Now, by hypothesis, $\semvalue{\f}^\model{N}_w= 1_\alg{A}$, so, by the Switch Lemma,  $\tuple{\model{N}^*, w} \models T^{1_\alg{A}}(\f)$. Then, $\tuple{\model{N}^*|k, w} \models T^{ 1_\alg{A}}(\f)$, and so $\tuple{\model{M}|k, w} \models T^{1_\alg{A}}(\f)$, as desired.

In order to prove (2), suppose that  $\f^*$ is  \alg{2}-valid in a frame $\mathfrak{F}$. Take any model for $\fm_\alg{2}^{\Diamond \Box}(\tau^*)$ of the form $\model{M}^*$ based on  $\mathfrak{F}$ for some corresponding model $\model{M}$ for $\fm_\alg{A}^{\Diamond \Box}(\tau)$ also based on $\mathfrak{F}$. Since $\f^*$ is globally true in $\model{M}^*$ (and all the antecedents of $\f^*$ hold at any world), we must have that  $T^{1_\alg{A}}(\f)$ is globally true in $\model{M}^*$, and by the Switch Lemma, $\f$  is globally true in $\model{M}$. Since for any such  $\model{M}$ based on $\mathfrak{F}$ there is a corresponding $\model{M}^*$, it follows that $\f$ is $\alg{A}$-valid in $\mathfrak{F}$. Thus, $\mathfrak{F} \in \mathbb{F}$.

($\Leftarrow$): Assume now that $\f^*$ modally $\alg{2}$-defines the class $\mathbb{F}$. As before, we need to show the following two claims:
\begin{itemize}
\item[(1)] the formula  $\f$ is \alg{A}-valid in every frame in $\mathbb{F}$, and
\item[(2)] every frame where $\f$ is \alg{A}-valid is in the class $\mathbb{F}$.
\end{itemize}
To see (1), consider any \alg{A}-valued model $\model{M}$ based on a frame $\mathfrak{F} \in \mathbb{F}$. In the corresponding model $\model{M}^*$ for $\fm_\alg{2}^{\Diamond \Box}(\tau^*)$ then $T^{1_\alg{A}}(\f)$ is globally true (as $\f^*$ is), so by the Switch Lemma, $\f$ is globally true in  $\model{M}$, as desired. To see (2), suppose now that $\f$ is \alg{A}-valid in the frame $\mathfrak{F}$. Reasoning as before, by the Switch Lemma, this means that $\f^*$  is $\alg{2}$-valid in $\mathfrak{F}$, so $\mathfrak{F} \in \mathbb{F}$.
\end{proof}

%\begin{Rmk}\em{
%In the case of many-valued frames, we can obtain a similar result. Before stating it, consider, for any class $\mathbb{F}$  of $\alg{A}$-valued frames, a corresponding class $\mathbb{F}^*$ of polymodal frames containing exactly the frames that come from $\alg{A}$-valued frames in $\mathbb{F}$ by means of the frame transformation described in \S \ref{polymodal}. Observe that from any frame in $\mathbb{F}^*$ one can recover uniquely the corresponding frame in $\mathbb{F}$.
%Now  let \alg{A} be a finite lattice algebra,  $\f$ a formula  from $\fm_\alg{A}^{\Diamond \Box}(\tau)$ for some finite $\tau$, and  $\mathbb{F}$ a class of $\alg{A}$-valued frames. Then, $\f$ modally \alg{A}-defines $\mathbb{F}$ iff\/  $\mathbb{F}^*$ is modally $\alg{2}$-defined by the polymodal formula $T^{1_\alg{A}}(\f) \in \fm_\alg{2}^{A}(\tau^*)$.

%$$\f^*:= (\bigvee_{m \leq \mathrm{rank}(T^{1_\alg{A}}(\f))} \neg \bigwedge_{b\in A} \Box^m_a ( \bigwedge T^*(\tau) ) )  \vee T^{1_\alg{A}}(\f). $$
% }
% \end{Rmk}

\subsection{When does modal \alg{2}-definability imply modal \alg{A}-definability?}\label{ClassicalImpliesModal}
The aim of this subsection is to provide sufficient conditions for recovering the modal \alg{A}-definability of a class of frames from its modal \alg{2}-definability. Whether the conditions we provide are necessary or not is left as an open problem. In Theorem~\ref{t:AdefImplies2def} we have seen that, given a finite lattice algebra $\alg{A}$, any class of (crisp) frames definable by a formula of $\fm_\alg{A}^{\Diamond \Box}$ in the $\alg{A}$-valued associated modal logic  is definable by a formula\ of classical modal logic. To have the reciprocal, one would expect the algebra $\alg{A}$ to interpret classical logic in some sense.

Assume for the rest of this section that $\alg{A}$ has the following non-trivial lattice reduct\footnote{For the sake of lighter notation, in this section, we often drop the superindex $\alg{A}$ in the operations.} $$\text{Red}(\alg{A})=\tuple{A,\wedge ,\vee ,\0,\1,\lnot}$$ enriched with bottom element $\0$, top element $\1$, and $\lnot x= u(x)$, (where $u(x)$ is a distinguished unary term). An example is when $\alg{A}$ is a bounded residuated lattice and $\lnot x=x\rightarrow \0$. However, in general, we do not require $\alg{A}$ to be residuated.

In this context, we provide  a sufficient algebraic condition for formulas  in the language $\fm_\alg{A}^{\Diamond \Box}$ to \alg{A}-define any class of frames definable in classical modal logic. We start by introducing a useful definition:

\begin{Def} A $\{\vee ,\wedge ,\1,\lnot \}$-algebra $\alg{B}$ is said to be \emph{interpretable in} $\alg{A}$ \emph{via a unary} $\fm_{\alg{A}}$\emph{-term} $t(x)$
if 
\begin{itemize}
\item $Eq(t^{\alg{A}})=\{\tuple{x,y}\in A^{2} \mid t^{\alg{A}}(x)=t^{\alg{A}}(y)\}$ is a congruence of $\text{Red}(\alg{A})$, 

\item $t^{\alg{A}}(\1^{\alg{A}})=\1^{\alg{A}}$, and

\item $\text{\emph{Red}}(\alg{A})/Eq(t^{\alg{A}})\,$ is isomorphic to $\alg{B}$.
\end{itemize}

Equivalently, $t(x)\wedge ^{\alg{B}'}t(y):=t^{\alg{A}}(x\wedge y),$ $t(x)\vee
^{\alg{B}'}t(y):=t^{\alg{A}}(x\vee y),$ \ $\lnot ^{\alg{B}'}t(x):=t^{\alg{A}}(\lnot x)$ are well-defined operations in $t^{\alg{A}}(A),$ and $\alg{B}$ is isomorphic to the algebra
\begin{equation*}
\alg{B'}=\tuple{t^{\alg{A}}(A),\wedge^{\alg{B}'},\vee ^{\alg{B}'},\lnot ^{\alg{B}'},\1^{\alg{A}}}.
\end{equation*}
\end{Def}

Clearly, $t^{\alg{A}} \colon A \rightarrow B'$ is an epimorphism and $\alg{B}'$ is a lattice with top element $\1^\alg{A}\!.$ If $\alg{A}$ is bounded, then so is $\alg{B}'$ and $\0^{\alg{B}'}=t^\alg{A}(\0^{\alg{A}}).$

\begin{Exm}\label{ex5}\em{
Any pseudocomplemented lattice $\alg{A}$ interprets
via $t(x)=\lnot \lnot x$ its algebra of \emph{regular elements} $\mathrm{Reg}(\alg{A})=\{x\in A\mid\lnot \lnot x=x\},$ which happens to be a Boolean algebra.
 To see this, notice that $\mathrm{Reg}(\alg{A})=\{\lnot \lnot x \mid x \in A\}$ because $\lnot \lnot \lnot
\lnot x=\lnot \lnot x,$ and $\lnot \lnot \1=\1.$ Moreover, $\lnot \lnot
x=\lnot \lnot y$ is a congruence since $\lnot \lnot x=\lnot
\lnot x^{\prime }$  and $\lnot \lnot y=\lnot
\lnot y^{\prime }$ imply: \medskip

$\lnot \lnot (x\wedge y)=\lnot \lnot (\lnot \lnot x\wedge \lnot \lnot
y)=\lnot \lnot (\lnot \lnot x^{\prime }\wedge \lnot \lnot y^{\prime })=\lnot
\lnot (x'\wedge y'),$

$\lnot \lnot (x\vee y)=\lnot \lnot (\lnot \lnot x\vee \lnot \lnot y)=\lnot
\lnot (\lnot \lnot x^{\prime }\vee \lnot \lnot y^{\prime })=\lnot \lnot
(x'\vee y').$ \medskip

\noindent Therefore, in $B=\mathrm{Reg}(\alg{A})$ we have the following induced operations: \medskip

$x\wedge ^{\alg{B}}y:=\lnot \lnot (x\wedge y)=\lnot \lnot x\wedge \lnot \lnot y =x\wedge y$

$x\vee ^{\alg{B}}y:=\lnot \lnot (x\vee y)\ $(no further reduction is possible)

$\lnot ^{\alg{B}}x:=\lnot \lnot (\lnot x)=\lnot x$

$\1^{\alg{B}}:=\1,$ $\0^{\alg{B}}:=\lnot \lnot \0=\0.$\medskip

\noindent and $\mathrm{Reg}(\alg{A})$ is a Boolean algebra because $x\wedge ^{\alg{B}}\lnot
x=\lnot \lnot \0=\0$ and $x\vee ^{\alg{B}}\lnot x=\lnot \lnot (x\vee \lnot x)=\1$ by the
density of $x\vee \lnot x$ in pseudocomplemented lattices (see \citep{B 2006}
for the identities utilized). In particular, any Heyting algebra \alg{A} interprets a Boolean algebra in this way.
}
\end{Exm}

\begin{Rmk}\em{ In Example \ref{ex5}, $\mathrm{Reg}(\alg{A})$ is not necessarily a subalgebra of $\alg{A}$ (because of disjunction), but it contains as a subalgebra its \emph{Boolean skeleton} ${\bf B}(\alg{A})$. These algebras coincide if and only if $\alg{A}$ is a Stone algebra (see~\citep{C-D-T 2015}).}\end{Rmk}

\begin{Exm}\label{ex6}\em{Any finite \MV-algebra $\alg{A}$ interprets its
Boolean skeleton ${\bf B}(\alg{A}),$ via $t(x)=nx$ for any $n\geq |A|$, because $n\alg{A}={\bf B}(\alg{A})$
and $n(\cdot) \colon \alg{A}\rightarrow \alg{A}$ is an endomorphism thanks to the validity of the following equations:\medskip

$n(x\wedge y) \approx nx\wedge ny$

$n(x\vee y) \approx nx\vee ny$

$n(\lnot x) \approx \lnot nx$

$n\1 \approx\1,$ $n\0 \approx\0.$\medskip

\noindent To see this, consider a subdirect representation $\alg{A}\subseteq \Pi
_{i\in I}\alg{C}_{i}$, where the $\alg{C}_{i}$ are finite \MV-chains. Each chain has length at most $n$ and thus it is easily verified by cases that $nx\in \{\0^{\alg{C}_i},\1^{\alg{C}_i}\}$ in each chain $\alg{C}_{i}$; therefore $n\alg{A}\subseteq {\bf B}(\alg{A}).$ Moreover, if $x=\tuple{x_{i}}_{i}\in {\bf B}(\alg{A})$,
then $x_{i}\in {\bf B}(\alg{C}_i)=\{\0^{\alg{C}_i},\1^{\alg{C}_i}\}.$ Therefore, $nx_{i}=x_{i}$ for each $i \in I$, and thus $x=nx\in n\alg{A}.$}

Notice that if $\alg{A}$ is a Glivenko bounded residuated lattice in the sense of \citep{C-T 2004}, then $\alg{A}$ interprets $\mathrm{Reg}(\alg{A})$ via $\lnot \lnot$. Although this algebra is not necessarily Boolean, in some cases it is an \MV-algebra and, thus, it interprets a Boolean algebra. 
\end{Exm}

\begin{Exm}\label{ex7}\em{By  \citep[Thm. 2]{tu}, in every \BL-algebra one can define an  \MV-algebra, and then, in turn, one can interpret a Boolean algebra as remarked before.}\end{Exm}

We are interested in term interpretations of Boolean algebras due to the following result:

 \begin{Thm}\label{lem1} Let \alg{A} be a finite lattice algebra which interprets
via a term t(x) a Boolean algebra, let  $\f$ be a formula from $\fm_\alg{2}^{\Diamond \Box}(\tau)$ and let $\mathbb{F}$ be a class of frames.
Then, $\f$ modally \alg{2}-defines $\mathbb{F}$ \ iff \ $t(\f)$  modally $\alg{A}$-defines $\mathbb{F}$.
 \end{Thm}

Then, putting this together with Theorem~\ref{t:AdefImplies2def}, we obtain:

\begin{Cor}\label{c:SameDefinableFrames}
Let \alg{A} be a finite lattice algebra which interprets via a term $t(x)$ a Boolean algebra. Then, the class of modally \alg{A}-definable frames coincides with the class of modally \alg{2}-definable frames.
\end{Cor}

\begin{Exm} \em{Thanks to the previous examples, Corollary~\ref{c:SameDefinableFrames} implies that the modal logic associated to a finite pseudocomplemented lattice or to a finite \MV-algebra defines the same class of
crisp frames as classical modal logic. In particular, this result includes the modal extensions of \L ukasiewicz finitely-valued logics studied by Teheux in \citep{te} and solves the open problem left in that paper of determining whether their definable classes of frames coincide with those definable in classical modal logic.}
\end{Exm}

The rest of this subsection is devoted to proving Theorem~\ref{lem1}. We start by showing that the two-valued global consequence of modal formulas is preserved when allowing models to take values on an arbitrary finite Boolean algebra.

\begin{Pro}\label{p:FiniteBoolean}
Let $\mathfrak{F}$ be a frame and $\alg{B}$ be a finite Boolean algebra. Then, for each set of formulas $\Gamma \cup \{\f\} \subseteq \fm_\alg{2}^{\Diamond \Box}(\tau)$, we have:
\[
\Gamma \vDash_{\mathrm{Log}(\{\mathfrak{F}\},\alg{B},\tau)} \f \text{\, if and only if }
\text{\ }\Gamma \vDash_{\mathrm{Log}(\{\mathfrak{F}\},\alg{2},\tau)} \f .
\]
In particular, when $\Gamma = \emptyset$, we obtain that $\f$  modally $\alg{2}$-defines and $\alg{B}$-defines the same class of frames.
\end{Pro}

\begin{proof}
Without loss of generality, we may identify $\alg{B}$ with a power $\alg{2}^{n}$ of the two-element algebra, and thus any valuation $V \colon \tau\times W\rightarrow B$ has the form $V(p,w)=\tuple{V_i(p,w)}_{i=1}^n$ where $V_i\colon \tau\times W\rightarrow 2$. Take a model $\model{M}= \tuple{\mathfrak{F},V}$. An easy induction shows that for any formula $\p$ and any $w\in W,$
\[
\semvalue{\p}_{w}^\model{M}=\tuple{\semvalue{\p}_{w}^{\tuple{\mathfrak{F},V_i}}}_{i=1}^n.
\]
Assume that $\Gamma \vDash_{\mathrm{Log}(\{\mathfrak{F}\},\alg{2},\tau)} \f$ and $\Gamma$ is globally true in a $\alg{B}$-valued model $\model{M}= \tuple{\mathfrak{F},V}$. That is, for each $\p \in \Gamma$ and each $w \in W$,  $\semvalue{\p}_{w}^\model{M}=1_\alg{B}=\tuple{1}_{i =1}^n$, and hence $\semvalue{\p}_{w}^{\tuple{\mathfrak{F},V_i}}=1$. Therefore, $\semvalue{\f}_{w}^{\tuple{\mathfrak{F},V_i}}=1$ for each $i \in \{1,\ldots, n\}$ and each $w\in W$, i.e.\ $\f$ is globally true in $\model{M}$ as desired. Reciprocally, assume that $\Gamma \vDash_{\mathrm{Log}(\{\mathfrak{F}\},\alg{B},\tau)} \f $ and $\Gamma$ is globally true in a classical model $\model{M}= \tuple{\mathfrak{F},V}$. Then, the diagonal valuation $V^{\prime}(p,w)=\tuple{V(p,w)}_{i=1}^n$ only gives values in $\{\tuple{1}_{i =1}^n,\tuple{0}_{i =1}^n\}\subseteq B$ and is such that $\semvalue{\p}_{w}^{\tuple{\mathfrak{F},V'}}=\tuple{1}_{i =1}^n = 1_\alg{B}$ for each $\p \in \Gamma$ and each $w \in W$. By hypothesis, $\semvalue{\f}_{w}^{\tuple{\mathfrak{F},V'}}=\tuple{\semvalue{\f}_{w}^{\tuple{\mathfrak{F},V}}}_{i=1}^n=\tuple{1}_{i =1}^n$ for each $w \in W$, and thus $\semvalue{\f}_{w}^{\model{M}}=1$ for each $w \in W$.
\end{proof}

The next step consists in obtaining a similar preservation result when changing the algebra via the notion of interpretability defined above.
 
\begin{Pro}\label{key}
Let $\mathfrak{F}$ be a frame and assume that a finite lattice algebra $\alg{A}$ interprets a $\{\vee ,\wedge ,\1,\lnot \}$-algebra $\alg{B}$ via a unary term $t(x)$. Then, for each set of formulas $\Gamma \cup \{\f\} \subseteq \fm_\alg{2}^{\Diamond \Box}(\tau)$, we have:
\begin{equation*}
\Gamma \vDash_{\mathrm{Log}(\{\mathfrak{F}\},\alg{B},\tau)}\f \,\,\, \text{iff} \,\,\, t[\Gamma] \vDash_{\mathrm{Log}(\{\mathfrak{F}\},\alg{A},\tau)}t(\f).
\end{equation*}
In particular, when $\Gamma = \emptyset$, we obtain that the class of frames modally $\alg{B}$-defined by $\f$ coincides with the class of frames modally $\alg{A}$-defined by $t(\f)$.
\end{Pro}

\begin{proof}
For any valuation $V \colon \tau\times W \rightarrow A$ (forming an \alg{A}-valued model $\model{M}$ based on $\mathfrak{F}$), let $V^{\ast }=t^{\alg{A}}\circ V \colon \tau\times W\rightarrow B$ (forming a \alg{B}-valued model $\model{M}' $ based on $\mathfrak{F}$).  Then, one can show by induction on the complexity of formulas $\p \in \fm_\alg{2}^{\Diamond \Box}(\tau)$ that for any $w\in W$:
\begin{equation*}
\semvalue{\p}^{\model{M}'}_w=t^{\alg{A}}(\semvalue{\p}^{\model{M}}_w)=   \semvalue{t(\p)}^{\model{M}}_w\in B.\text{ }
\end{equation*}
Assume that $\Gamma \vDash_{\mathrm{Log}(\{\mathfrak{F}\},\alg{B},\tau)}\f$ and $t[\Gamma]$ is globally true in an $\alg{A}$-valued model $\model{M}= \tuple{\mathfrak{F},V}$. Then, by the previous observation, $\Gamma$ is globally true in the $\alg{B}$-valued model $\model{M}'$ and, hence, so is $\f$. Therefore, $t(\f)$ is globally true in $\model{M}$.

Reciprocally, assume that $t[\Gamma] \vDash_{\mathrm{Log}(\{\mathfrak{F}\},\alg{A},\tau)}t(\f)$. Suppose now that $\Gamma$ is globally true in a $\alg{B}$-valued model $\model{N} = \tuple{\mathfrak{F},V_B}$. Choose an \alg{A}-valued $\model{M}$ with a valuation $V\colon \tau \times
W\rightarrow A$ such that $V_{B}=t^{\alg{A}}\circ V$ and hence $\model{N}=\model{M}'$. Then, $t[\Gamma]$ is globally true in $\model{M}$ and, hence, so is $t(\f)$. Therefore, $\f$ is globally true in $\model{N}$ as desired.
\end{proof}

Now, it is clear that Theorem~\ref{lem1} follows from Propositions~\ref{p:FiniteBoolean} and~\ref{key}. Furthermore, since the preservation of consequence in these propositions holds at the level of a fixed frame, it is clear that it also holds for the consequence given by a class of frames:

\begin{Cor}\label{c:Translation-Classical-to-ManyValued}
Let $\mathbb{F}$ be a class of frames and assume that a finite lattice algebra $\alg{A}$ interprets a Boolean algebra $\alg{B}$ via a unary term $t(x)$. Then, for each set of formulas $\Gamma \cup \{\f\} \subseteq \fm_\alg{2}^{\Diamond \Box}(\tau)$, we have:
\begin{equation*}
\Gamma \vDash_{\mathrm{Log}(\mathbb{F},\alg{2},\tau)}\f \,\,\, \text{iff} \,\,\, t[\Gamma] \vDash_{\mathrm{Log}(\mathbb{F},\alg{A},\tau)}t(\f).
\end{equation*}
\end{Cor}

We end this subsection with two technical remarks regarding the notion of interpretability of algebras that has been instrumental in our translation.

\begin{Rmk} \em{The property of interpreting an algebra $\alg{B}$ in $\alg{A}$ via a term $t(x)$ is quasiequational in $\alg{A}$. Hence, if $\alg{A}$ interprets a Boolean
algebra via $t(x)$, then any member of \/ $\mathbf{Q}(\alg{A})$ interprets a Boolean
algebra via $t(x)$ (not necessarily the same algebra). If $t(x)$ is idempotent (which is
the case in the examples we have), this property has a simple equational
characterization because the identities in $\alg{A}$: \medskip

$t(x\vee y)\approx t(t(x)\vee t(y))$

$t(x\wedge y) \approx t(t(x)\wedge t(y))$

$t(\lnot x) \approx t(\lnot t(x))$

$t(\1) \approx \1$

\noindent imply the congruence character of $Eq(t).$ They are actually
equivalent to idempotency plus congruence. Therefore, if $\alg{A}$ interprets $\alg{B}$
via an idempotent term $t,$ then any member of \/ $\mathbf{V}(\alg{A})$ interprets a
member of \/ $\mathbf{V}(\alg{B})$ via $t.$}  \end{Rmk}

\begin{Rmk} \em{The conditions for $nx$ in Example~\ref{ex6} are
equational, except the inclusion ${\bf B}(\alg{A})\subseteq n\alg{A}$ which is given by a
quasiequation. Hence, any algebra in the variety $\mathbf{V}(\alg{A})$ generated
by $\alg{A}$ interprets via $x\longmapsto nx$ a\ subalgebra of its Boolean
skeleton, and any algebra in the quasivariety $\mathbf{Q}(\alg{A})$ interprets the
full Boolean skeleton.

According to \citep{C-T 2012}, all algebras of a variety $\mathbb{V}$ of\/
\MV-algebras interpret via a term their full skeleton if and only if $\mathbb{V}$
satisfies the equation $2x^{2} \approx (2x)^{2},$ in which case $t(x)=2x^{2}$ does
the job. This is the case of the variety generated by the \emph{Chang algebra}. However, this example is orthogonal to ours because the only non-trivial finite \MV-algebra satisfying this equation is $\alg{2}.$
}
\end{Rmk}

%If $\mathbb{F}$ is a class of models in $\fm_\alg{A}^{\Diamond \Box}(\tau)$, let $\mathbb{F}^*=\{\model{M}^* \mid \model{M}\in \mathbb{F}\}$, a class of models in in $\fm_\alg{2}^{\Diamond \Box}(\tau^*)$, we can obtain:
% 
% 
% \begin{Lem}\label{lem3} 
%The formula $\f$ from $\fm_\alg{A}^{\Diamond \Box}(\tau)$ (assume w.l.o.g. by the finite occurrence property that $\tau$ is the finite set of variables $\{p_1, \ldots, p_n\}$ of $\f$),  modally \alg{A}-defines the model class $\mathbb{F}$ iff  in $\fm_\alg{2}^{\Diamond \Box}(\tau^*)$ $\mathbb{F}^*$ is modally $\alg{2}$-defined by $T^1(\f)$.
%
%%$$  (\bigwedge_{\substack{b\in A}}p_0^b)\wedge (\bigwedge_{0<i \leq n}\bigvee_{\substack{a \in A }}p_i^a)  \wedge (\bigwedge_{\substack{a, b \in A \\ a \neq b\\ 0<i \leq n}} (\neg p_i^a\vee \neg p_i^b))    \wedge T^1(\f).$$
% \end{Lem}
% 
% 
%\begin{Lem}\label{lem4}  The formula $\f$ from $\fm_\alg{2}^{\Diamond \Box}(\tau^*)$ modally $\alg{2}$-defines the class of models $\mathbb{F}^*$ iff the formula $T' (\f) $ (see the proof of Theorem \ref{ch}) from $\fm_\alg{A}^{\Diamond \Box}(\tau)$  modally $\alg{A}$-defines the class of models $\mathbb{F}$. 
%
% \end{Lem}
%
%Notice that in Lemmas \ref{lem3} and \ref{lem4}, the language need not have all the truth-constants we were working with so far: they are quite unnecessary for the argument to go through.

\subsection{Goldblatt--Thomason Theorem and related results} 
With the main theorems in hand, we are already in a position to state  a  Goldblatt--Thomason theorem for a large class of many-valued modal logics as a consequence of the classical theorem itself~\citep[Thm.~8]{gold} and Corollary~\ref{c:SameDefinableFrames}.
 
\begin{Cor}[Finitely-valued Goldblatt--Thomason Theorem]\label{gt} Let \alg{A} be a finite lattice algebra which interprets
via some term  a Boolean algebra. Furthermore, let $\mathbb{F}$ be an elementary class of frames. Then, $\mathbb{F}$ is modally \alg{A}-definable by a set of formulas  in $\fm_\alg{A}^{\Diamond \Box}(\tau)$ iff\/ $\mathbb{F}$ is closed under taking generated subframes, disjoint unions, and bounded morphic images, and reflects ultrafilter extensions. \end{Cor}

%\begin{proof} 
%($\Leftarrow$) Suppose now that $\mathbb{F}$ is closed under taking generated subframes, disjoint unions, bounded morphic images, and reflects ultrafilter extensions. By Theorem 8 \citep{gold}, we have that $\mathbb{F}$ is modally $\alg{2}$-definable (in $\fm_\alg{2}^{\Diamond \Box}(\tau)$) by a set of formulas $\Phi$, and, hence, by Theorem~\ref{lem1}, $\mathbb{F}$ is  modally $\alg{A}$-definable in $\fm_\alg{A}^{\Diamond \Box}(\tau)$, as desired.
%
%($\Rightarrow$) If $\mathbb{F}$ is modally \alg{A}-definable (in $\fm_\alg{A}^{\Diamond \Box}(\tau)$) by a set of formulas $\Phi$, then $\mathbb{F}$ is the intersection of the classes modally \alg{A}-defined by each $\f \in \Phi$. Consequently, by~Theorem \ref{t:AdefImplies2def},
%$\mathbb{F}$ is modally $\alg{2}$-definable in $\fm_\alg{2}^{\Diamond \Box}(\tau^*)$, which means it is modally axiomatizable in the sense of \citep{gold}, so by applying the easy direction of  Theorem 8 \citep{gold}, we get that $\mathbb{F}$ is closed under taking generated subfr ames, disjoint unions, bounded morphic images, and reflects ultrafilter extensions. 
%\end{proof}

It is known that the characterizing conditions of Corollary~\ref{gt} could be weakened to closure under \emph{ultrapowers} or under \emph{ultrafilter extensions} (see \citep{gold2}). Moreover, from the two theorems obtained by Van Benthem in~\citep[Section~4.2]{van2}, we obtain the following additional characterizations of finite (transitive) definable frames:
 
\begin{Cor} Let \alg{A} be a finite lattice algebra which interprets
via some term a Boolean algebra. Then, a class $\mathbb{F}$ of finite frames is modally \alg{A}-definable  in $\fm_\alg{A}^{\Diamond \Box}(\tau)$ iff it is closed under taking generated subframes, finite disjoint unions, and local p-morphic images.
\end{Cor}

\begin{Cor}
Let \alg{A} be a finite lattice algebra which interprets via some term a Boolean algebra. Then, a class $\mathbb{F}$ of finite transitive frames is modally \alg{A}-definable in $\fm_\alg{A}^{\Diamond \Box}(\tau)$ iff it is closed under taking generated subframes, finite disjoint unions, and p-morphic images.
\end{Cor}
  
Observe now that the translation in \S \ref{trans} may be extended to the first-order modal setting rather easily (after all, the semantics of $\Diamond$ and $\Box$ is similar to that of the quantifiers $\exists$ and $\forall$):
\begin{center}

\begin{tabular}{rcll}
%$T^a(\overline{b})$ & $:=$ & 
%$\begin{cases*}
%      \bot & if $a\neq b$ \\
%     p_0^b      & otherwise
%    \end{cases*}$ &
%\\[0.5ex]
$T^a(P^n_i)$ & $=$ & 
$P_i^{na}$  \,\,\,\, ($i \geq 1$)&
\\[0.5ex]
%$T^a(\0)$ & $=$ & $\0$
%\\[0.5ex]
%$T^a(\1)$ & $=$ & $\1$
%\\\vspace{0.4ex}
%$T^a(\p_1 \land \p_2)$ & $=$ & $\bigvee\limits_{\substack {b_1, b_2 \in A \\b_1\land^\alg{A} b_2 = a}}(T^{b_1}(\p_1) \wedge  T^{b_2}(\p_2))$ & 
%\\\vspace{0.4ex}
%$T^a(\p_1 \lor \p_2)$ & $=$ & $\bigvee\limits_{\substack {b_1, b_2 \in A \\b_1\lor^\alg{A} b_2 = a}}(T^{b_1}(\p_1) \wedge  T^{b_2}(\p_2))$ & 
%\\\vspace{0.4ex}
%$T^a({\circ}(\vectn{\p}))$ & $=$ & $\bigvee\limits_{\substack {b_1,  \ldots, b_n \in A \\{\circ}^\alg{A}(b_1, \ldots, b_n) = a}}(T^{b_1}(\p_1) \wedge \ldots\wedge T^{b_n}(\p_n))$ & 
\\\vspace{0.4ex}
$T^a(\exists x\, \p)$ & $=$ & $(\bigvee\limits_{\substack {k \leq |A| \\b_1  \ldots b_k \in A \\ b_1 \vee^\alg{A} \ldots \vee^\alg{A} b_k = a}} \bigwedge\limits_{i=1}\limits^{k} \exists x\,  T^{b_i}(\p) ) \wedge \forall y\, ( \bigvee\limits_{\substack {b \in A \\{b \leq a}}}     T^{b}(\p(y/x)))$
\\\vspace{0.4ex}
$T^a(\forall x\, \p)$ & $=$ & $(\bigvee\limits_{\substack {k \leq |A| \\b_1,  \ldots, b_k \in A \\ b_1 \wedge^\alg{A} \ldots \wedge^\alg{A} b_k = a}} \bigwedge\limits_{i=1}\limits^{k} \exists x\,  T^{b_i}(\p) ) \wedge \forall y\, ( \bigvee\limits_{\substack {b \in A \\{a \leq b}}}     T^{b}(\p(y/x))). $

\end{tabular}
\end{center}
Consequently, using \citep[Thm.~3.6]{zho} (a Goldblatt--Thomason Theorem for first-order modal logic) and our method we could obtain:
  
\begin{Cor}
Let \alg{A} be a finite lattice algebra which interprets via some term  a Boolean algebra. Let\/ $\mathbb{F}$ be  a  class of frames closed under elementary equivalence. Then, $\mathbb{F}$ is modally \alg{A}-definable by a set of formulas  in $\fm_\alg{A}^{\Diamond \Box \exists\forall }(\tau)$ iff it is closed under bounded morphic images, taking generated subframes, and disjoint unions.
\end{Cor}
  
\begin{Rmk} \em{Observe that, in all the results, the hypothesis that $\alg{A}$ interprets via some term a Boolean algebra could be substituted by requiring $\alg{A}$ interprets via some term a Boolean algebra {\em or a Heyting algebra, or an MV-algebra, or a pseudocomplemented lattice}. This is clear by the results of the previous section.}
\end{Rmk}

\section{Applications to computational complexity}\label{comp}

Following~\citep[Section 3.3]{de} we may rewrite our translation from \S\ref{trans}, using additional propositional letters, in such a way that the length of the translated formula becomes polynomial on the length of the original one. However, this trick was defined only for the case of algebraic connectives, so we are going to extend it here to modal operators. This technique will allow us to easily check that the complexity of the consequence and validity problems of many-valued modal logics coincides with that of their two-valued counterparts.

Let now $\tau^* = \{q_{\f}^a \mid a \in A, \f \in \fm_\alg{A}^{\Diamond \Box}(\tau)\}$ and redefine $T^*(\tau)$ as \[
 \bigvee_{\substack{a \in A }}q_{p_i}^a, \,\,  \neg (q_{p_i}^a\land q_{p_i}^b) \,\,\,\,\,\, \,\,\,\,\,\,  (a, b \in A, a \neq b, p_i \in \tau). 
 \]

Then, we consider the theory $T^*(\tau) \cup \{E(\f) \mid  \f \in \fm_\alg{A}^{\Diamond \Box}(\tau)\}$ where (we use $\Rightarrow$ to represent the usual material implication and $\Leftrightarrow$ for material equivalence):

\begin{center}
\begin{tabular}{rcll}
%$T^a(\overline{b})$ & $:=$ & 
%$\begin{cases*}
%      \bot & if $a\neq b$ \\
%     p_0^b      & otherwise
%    \end{cases*}$ &
%\\[0.5ex]
%$T^a(\0)$ & $=$ & $\0$
%\\[0.5ex]
%$T^a(\1)$ & $=$ & $\1$
%\\\vspace{0.4ex}
%$T^a(\p_1 \land \p_2)$ & $=$ & $\bigvee\limits_{\substack {b_1, b_2 \in A \\b_1\land^\alg{A} b_2 = a}}(T^{b_1}(\p_1) \wedge  T^{b_2}(\p_2))$ & 
%\\\vspace{0.4ex}
%$T^a(\p_1 \lor \p_2)$ & $=$ & $\bigvee\limits_{\substack {b_1, b_2 \in A \\b_1\lor^\alg{A} b_2 = a}}(T^{b_1}(\p_1) \wedge  T^{b_2}(\p_2))$ & 
%\\\vspace{0.4ex}
$E({\circ}(\vectn{\p}) := $ & $q^a_{{\circ}(\vectn{\p})}$ & $\Leftrightarrow$ & $\bigvee\limits_{\substack {b_1,  \ldots, b_n \in A \\{\circ}^\alg{A}(b_1, \ldots, b_n) = a}}(q^{b_1}_{\p_1} \wedge \ldots\wedge q^{b_n}_{\p_n})$ 
\\\vspace{0.4ex}
$E(\Diamond \p) := $ & $q^a_{\Diamond \p}$ & $\Leftrightarrow$ & $(\bigvee\limits_{\substack {k \leq |A| \\b_1,  \ldots, b_k \in A \\ b_1 \vee^\alg{A} \ldots \vee^\alg{A} b_k = a}} \bigwedge\limits_{i=1}\limits^{k} \Diamond  q^{b_i}_{\p} ) \wedge \Box ( \bigvee\limits_{\substack {b \in A \\{b \leq a}}}     q^{b}_{\p})$
\\\vspace{0.4ex}
$E(\Box \p) := $ & $q^a_{\Box \p}$ & $\Leftrightarrow$ & $(\bigvee\limits_{\substack {k \leq |A| \\b_1,  \ldots, b_k \in A \\ b_1 \wedge^\alg{A} \ldots \wedge^\alg{A} b_k = a}} \bigwedge\limits_{i=1}\limits^{k} \Diamond  q^{b_i}_{\p} ) \wedge \Box ( \bigvee\limits_{\substack {b \in A \\{a \leq b}}}     q^{b}_{\p}). $

\end{tabular}
\end{center}

Observe that the length of these formulas is always bounded by $c2^{|A|}|A|$ where $c$ is a constant number.

Using this theory, we can obtain a translation from many-valued to classical consequence (that is, the reverse direction of Corollary~\ref{c:Translation-Classical-to-ManyValued} in a much less straightforward manner):

\begin{Thm}\label{t:Translation-ManyValued-to-Classical}
Let $\alg{A}$ be a finite lattice algebra, $\tau$ a denumerable set of variables, and\/ $\mathbb{F}$ a class of frames. Then, for each\/ $\Gamma \cup \{\f\} \subseteq \fm_\alg{A}^{\Diamond \Box}(\tau)$, we have:
\begin{center}
%\item $ \vDash_{\mathrm{Log}(\mathbb{F},\alg{A},\tau)} \f$ \,\, iff \,\, $ T^* \vDash_{\mathrm{Log}(\mathbb{F},\alg{2},\tau^*)} T^1(\f),$
$\Gamma \vDash_{\mathrm{Log}(\mathbb{F},\alg{A},\tau)} \f$ \,\, iff \,\, $\{q^{1_\alg{A}}_\theta \mid \theta \in \Gamma\}  \cup T^*(\tau) \cup \{E(\p) \mid  \p \in \fm_\alg{A}^{\Diamond \Box}(\tau)\} \vDash_{\mathrm{Log}(\mathbb{F},\alg{2},\tau^*)} q^{1_\alg{A}}_\f.$
%\item $\Gamma \vdash_{l\mathrm{Log}(\mathbb{F},\alg{A},\tau)} \f$ \,\, iff \,\, $T^1[\Gamma] \vdash_{l\mathrm{Log}(\mathbb{F},\alg{2},\tau^*)} T^1(\f)$
\end{center}
Furthermore, if\/ $\Gamma \cup \{\f\}$ is finite, the theory $T^*(\tau) \cup \{E(\p) \mid  \p \in \fm_\alg{A}^{\Diamond \Box}(\tau)\}$ can be taken to be finite as well, involving only the relevant axioms for $\p$ being a subformula of some formula of\/ $\Gamma \cup \{\f\}$. 
Similarly, as in Theorem~\ref{t:AdefImplies2def}, we have that
\begin{center}
%\item $ \vDash_{\mathrm{Log}(\mathbb{F},\alg{A},\tau)} \f$ \,\, iff \,\, $ T^* \vDash_{\mathrm{Log}(\mathbb{F},\alg{2},\tau^*)} T^1(\f),$
$\vDash_{\mathrm{Log}(\mathbb{F},\alg{A},\tau)} \f$ \,\, iff \,\, $ \vDash_{\mathrm{Log}(\mathbb{F},\alg{2},\tau^*)}(\bigwedge_{m \leq \mathrm{rank}(\f)}  \Box^m ( \bigwedge (T^*(\tau) \cup \{E(\p) \mid  \p \ \text{subformula of} \ \f\}) ) \Rightarrow q^{1_\alg{A}}_\f.$
%\item $\Gamma \vdash_{l\mathrm{Log}(\mathbb{F},\alg{A},\tau)} \f$ \,\, iff \,\, $T^1[\Gamma] \vdash_{l\mathrm{Log}(\mathbb{F},\alg{2},\tau^*)} T^1(\f)$
\end{center}
\end{Thm}

Note that the length of the formula in the last statement is polynomial w.r.t.\ the length of $\f$ (thanks, among others, to the bounded size observed above for the formulas $E(\p)$). Therefore, from Theorem~\ref{t:Translation-ManyValued-to-Classical} and the results from~\citep{chen1994} for consequence in two-valued modal logic, the problem of consequence from a finite set of premises in many-valued modal logics over all crisp frames is decidable. Moreover, is in EXPTIME for the logics K, T, and B, in PSPACE for the logic S4, and in co-NP for the logics KD45 and S5. Similarly, using the classical results from \citep{la}, we obtain that the problem of validity in many-valued modal logics over all crisp frames is in PSPACE for the logics K, T, B, and S4, and in co-NP for KD45 and S5. Finally, thanks to the reverse translation for the many-valued modal logics in Corollary~\ref{c:Translation-Classical-to-ManyValued}, we can conclude that all these computational problems are also complete in their corresponding complexity class. In particular, we have covered the complexity  results in \citep{bou2} for finitely-valued \L ukasiewicz modal logics through a completely different proof.

%Finally, our translation carries over to the setting of finitely-valued PDL presented in \citep{ig}, and hence, the complexity of the validity problem in finitely-valued PDL can be reduced to that of two-valued PDL, which is known to be in EXPTIME \citep{pratt}.  

\section{Conclusion}\label{con}
 
In this paper we have only scratched the surface of the potential of the translation introduced in \S \ref{trans}. 
We believe that the application in obtaining Corollary~\ref{gt} is  quite a nice illustration of the power of this translation. We did not only provide an alternative to the rather complex proof from~\citep{te}, but we also generalized the result to any finite residuated lattice that interprets a Boolean algebra by a term  (the case of finite \MV-algebras being just one example). Moreover, we also managed to prove some new Goldblatt--Thomason style results that were not considered in \citep{te}. Observe that, due to the duality of the modalities $\Diamond$ and $\Box$ classically (though not in our setting), the translation can be defined in the context of unimodal systems as well where we only have one of $\Diamond$ or $\Box$.   

A similar translation to that in  \S \ref{trans} can be offered for many-valued modal logics on frames with a many-valued relation by considering a suitable polymodal classical counterpart (we leave the details to the reader; the idea is to introduce classical modalities that will simulate the value of the accessibility relation). However, due to the added complexity introduced by the many-valued accessibility relation, we are not able to obtain an analogue of our main result with the help of such translation.

%In Theorem \ref{lem1}, even if we assumed the lattices to be
%complete, finiteness would still be necessary to handle the modal operators as there is no clear algebraic way to ensure that $t^{\alg{A}}$ preserves infima and suprema. 
%Moreover,

The more important open problem around Theorem \ref{lem1}, though, is whether all the conditions we have found are actually necessary. There are more general definitions of interpretability (see \citep{h}) because the homomorphism condition is needed only to handle the modal operators. It would be interesting to explore these more general versions. 

In future work, we intend to use the techniques in this paper to study other topics in finitely-valued modal logic, such as 0-1 laws. Furthermore, it is possible to reduce all  finitely-valued predicate logics to classical first-order logic through a translation and we intend to study the consequences (and limitations) of that  approach for model-theoretic studies of finitely-valued predicate logics.
 
Finally, we wish to remark that the results obtained here might be relevant for the philosophical debate~\citep{sc} around the so called ``Suszko's thesis"~\citep{su}, namely, that many-valued logics can be reduced to two-valued logic.

\section*{Acknowledgements}
\noindent Badia was partially supported by the Australian Research Council grant DE220100544. Badia and Noguera were also  supported by the European Union's Marie Sklodowska--Curie grant no.\ 101007627 (MOSAIC project).

\end{document}